\documentclass[a4paper,12pt]{amsart}

\usepackage{amssymb} 
\usepackage{xcolor} 
\usepackage{fancybox}
\usepackage{dashbox}
\newboolean{showcomments}
\setboolean{showcomments}{true}
\ifthenelse{\boolean{showcomments}}
{ \newcommand{\mynote}[3]{
  \fbox{\bfseries\sffamily\scriptsize#1}
  {\small$\blacktriangleright$\textsf{\emph{\color{#3}{#2}}}$\blacktriangleleft$}}}
{ \newcommand{\mynote}[3]{}}
\definecolor{asparagus}{rgb}{0.53, 0.66, 0.42}

  \usepackage{float}

  \usepackage{pgffor}
  \usepackage{tikz}
  \usetikzlibrary{calc,3d,arrows,positioning,shapes.misc}

  \usepackage[english]{babel}
  \usepackage{amsthm}
  \usepackage{amsmath}
  \usepackage{amssymb}
  \usepackage[latin1]{inputenc}
  \usepackage{setspace}
  \usepackage{enumerate}
  \usepackage{stmaryrd}

  \usepackage[colorlinks=true,linkcolor=black,citecolor=black,filecolor=black,urlcolor=black,menucolor=black]{hyperref}
  \usepackage{array}
  \usepackage{graphicx}
  \usepackage[babel]{csquotes}
  \usepackage{geometry}
  \usepackage{bbm}
  \usepackage{stmaryrd}
  \usepackage[all]{xy}
  \usepackage{mathrsfs}%für mathscr

\numberwithin{equation}{section}

  \theoremstyle{plain}
  \newtheorem{thm}{Theorem}[section]
  \newtheorem{prop}[thm]{Proposition}
  \newtheorem{cor}[thm]{Corollary}
  
  \newtheorem{lem}[thm]{Lemma}
  
  \theoremstyle{definition}
  \newtheorem{defn}[thm]{Definition}
  
  \newtheorem{rem}[thm]{Remark}

  \usepackage{xpatch}
  \xpatchcmd{\proof}{\itshape}{\normalfont\bfseries}{}{}

  \def \F{\mathcal{F}}
  \def \H{\mathcal{H}}
  \def \L{\mathcal{L}}
  \def \N{\mathbb{N}}
  \def \R{\mathbb{R}}

  \newcommand {\supp} {\mathop \textup{supp}}
  \def \chara{\mathbf{1}}
 \newcommand {\diver} {\mathop \textup{div}}
 \newcommand {\Int} {\mathop \textup{Int}}
 \newcommand {\dist} {\mathop \textup{dist}}
 \DeclareMathOperator{\aplimsup}{ap-limsup}

 \def \BV{{BV}}
 \newcommand{\lcorner}{%
  \,\raisebox{-.127ex}{\reflectbox{\rotatebox[origin=br]{-90}{$\lnot$}}}\,%
}

  \title[Implicit time discretization for the MCF of outward minimizing sets]{Implicit time discretization for the mean curvature flow of mean convex sets}
\author{Guido De Philippis}
\address{G.D.P.: Scuola Internazionale Superiore di Studi Avazanti, Via Bonomea 265, 34136 Trieste, Italy}
\email{guido.dephilippis@sissa.it}

\author{Tim Laux}
\address{T.L.: Department of Mathematics, University of California, Berkeley, CA 94720-3840 USA}
\email{tim.laux@math.berkeley.edu}
    \date{\today}

\begin{document}

    \begin{abstract}
    In this note we analyze the Almgren-Taylor-Wang scheme for mean curvature flow in the case of mean convex initial conditions. We show that the scheme  preserves strict mean convexity and, by compensated compactness techniques, that the arrival time functions converge strictly in \(BV\). In particular, this  establishes the convergence of the time-integrated perimeters of the approximations. As a corollary, the conditional convergence result of Luckhaus-Sturzenhecker becomes unconditonal in the mean convex case.
    
    \medskip
    

  \medskip

  \end{abstract}
\maketitle
\section{Introduction}

In 1993, Almgren-Taylor-Wang \cite{ATW93} proposed an implicit time discretization for mean curvature flow, which comes as a family of variational problems. Given an open subset $E_0\subset\R^n$ and a time-step size $h>0$, the sets $E_1,E_2,\ldots$ are successively obtained by solving
\begin{equation}\label{MM}
E_{k} \in \arg\min_E \Big\{ P(E)+ \frac1h \int_{E \Delta E_{k-1}} d_{E_{k-1}} \Big\},
\end{equation}
where $P(E) = \sup \{ \int_E \diver \xi \colon \|\xi\|_\infty \leq 1\}$ denotes the De Giorgi  perimeter of a subset of $\R^n$, $d_E$ the distance function to the boundary of $E$ and $E\Delta E_{k-1}$ the symmetric difference of $E$ and $E_{k-1}$.

\medskip

At the very heart of their idea lies the gradient-flow structure of mean curvature flow: trajectories in state space follow the steepest descent of the area functional with respect to an $L^2$-type metric.
In fact, this scheme inspired Ennio De Giorgi \cite{de1993new} to define his minimizing movements for general gradient flows in metric spaces, see \cite{ambrosio2006gradient}. Given a metric $\dist $ and an energy functional $E$, each time step of his abstract scheme is a minimization problem of the form
\[
x_{k} \in \arg \min_x \Big\{ E(x) + \frac1{2h} {\dist}^2(x,x_{k-1})\Big\}.
\]
In the smooth finite dimensional case when $\dist$ is the induced distance of a Riemannian metric, the Euler-Lagrange equation of the scheme boils down to the implicit Euler scheme.

\medskip

In case of mean curvature flow, the metric tensor ($L^2$-metric on normal velocities) is completely degenerate in the sense that the induced distance vanishes identically \cite{michor2004riemannian}. This  explains the use of the proxy $2\int_{E_{k+1}\Delta E_{k}} d_{E_k}$ for the squared distance in the minimizing movements scheme \eqref{MM}.

\medskip

The initial motivation of \cite{ATW93} was to define a generalized mean curvature flow through singularities as limits of the scheme \eqref{MM}. The convergence analysis as $h\downarrow0$ has a long history: Compactness of the approximate solutions was already established in \cite{ATW93}, together with the consistency of the scheme, in the sense that the approximations converge to the smooth mean curvature flow as long as the latter exists. In \cite{chambolle2004algorithm}, Chambolle simplified the proof and, furthermore,   proved  convergence to the viscosity solution (see \cite{evans1995motion}), provided the latter is unique. More precisely,  setting  $E_h(t) = E_k$, $t\in [kh, (k+1)h)$ to be the piecewise constant in  time interpolation of the sets $E_k$ obtained from \eqref{MM}, then the result reads as follows, see \cite{BarlesSonerSouganidis} for the notion of viscosity solution in this context.
\begin{thm}[{Convergence to viscosity solution \cite[Theorem 4]{chambolle2004algorithm}}]\label{thm Chambolle}
	Suppose $T<\infty$ and $E_0$ is a bounded  set in $\R^n$ with $\L^n(\partial E_0)=0$ such that the viscosity solution $\chara_{E(t)}$ starting from $\chara_{E_0}$ is unique, then $E_h\to E$ in $L^1$, i.e., $\int_0^T|E_h(t) \Delta E(t)|\,dt \to 0$ as $h\downarrow 0$.
\end{thm}

\medskip

Only shortly after \cite{ATW93}, Luckhaus-Sturzenhecker \cite{LucStu95} published a conditional convergence result which does not rely on the comparison principle but is purely based on the gradient-flow structure of mean curvature flow. In particular they showed that, \emph{conditioned on the convergence of the perimeters}, the scheme converges to a \(BV\) solution of mean curvature flow, according to the following definition.

\begin{defn}
A set of finite perimeter  \(E\subset \R_+\times \R^n\) is a \(BV\) solution of mean curvature flow if there exists \(V\in L^2(0,T;L^2(\H^{n-1}\lcorner \partial^\ast E(t)))\) such that
\begin{align}
&\int_0^T \int_{\partial^\ast E(t)} \left(\diver \xi - \nu \cdot D\xi \, \nu\right) d\H^{n-1}\,dt
= -\int_0^T \int_{\partial^\ast E(t)} V \, \xi \cdot \nu\, d\H^{n-1}\,dt,
\\\label{BV sol}
&\int_0^T\int_{E(t)}\partial_t \psi(t,x)\,dx\,dt+\int_{E(0)}\psi(0,x)\,dx=-\int_0^T \int_{\partial^\ast E(t)} \psi(t,x) V d\H^{n-1}(x)\,dt
\end{align}
for all \(\xi\in C^1_{c}([0,T)\times \R^n;\R^n)\) and \(\psi \in C^1_{c}([0,T)\times \R^n;\R)\). Here \(E(t)\) is the time slice of \(E\),  \(\partial^\ast\) denotes the reduced boundary, and $\nu$ the (measure theoretic) exterior normal.
\end{defn}

%If the initial datum $E_0$ is a set of finite perimeter, they establish compactness of the piecewise constant interpolations $E_h(t)$ of the iterations \eqref{MM}:
%\begin{prop}[{Compactness \cite[Lemma 1.2 and Proposition 1.6]{LucStu95}}]\label{prop LS}
%If $E_0$ is a bounded open set in $\R^n$ with $P(E_0)<\infty $ and $T<\infty$, then there exists a subsequence $h_j\downarrow 0$ and sets of finite perimeter $E(t) \subset \R^n$ such that $E_{h_j} \to E$ in $L^1$.
%\end{prop}  

The main result in \cite{LucStu95} is the following conditional convergence result:

\begin{thm}[{Conditional convergence \cite[Theorem 2.3]{LucStu95}}]\label{thm LS} Let $n\leq 7$ and let  \(E_h\) be the (time) piecewise constant approximation built by the Almgren-Taylor-Wang scheme. Then there exists a set \(E\subset \R_+\times \R^n\) and a subsequence \(\{h_j\}\) such that $E_{h_j}(t) \to E$ in $L^1$. Moreover, if
\begin{equation}\label{conv_ass}
\lim_{h_j \downarrow 0} \int_0^T P(E_{h_j}(t)) \,dt = \int_0^T P(E(t))\,dt,
\end{equation}
then $E$ is a $\BV$ solution of mean curvature flow. 
\end{thm}
%For completeness, we briefly recall the notion of $\BV$ solutions of mean curvature flow, which are distributional solutions in the following sense: For all (compactly supported) test vector fields $\xi=\xi(x,t)$
%\begin{equation}\label{BV sol}
%\int_0^T \int_{\partial^\ast E(t)} \left(\diver \xi - \nu \cdot D\xi \, \nu\right) d\H^{n-1}\,dt
%= -\int_0^T \int_{\partial^\ast E(t)} V \, \xi \cdot \nu\, d\H^{n-1}\,dt,
%\end{equation}
%where $\nu$ denotes the (exterior) normal to $\partial^\ast E(t)$ and $V$ is the (square integrable) normal  velocity of $\partial^\ast E(t)$ in the sense of distributions, i.e., $\partial_t \chi_{E(t)} = V \,\H^{n-1}\lcorner \partial^\ast E(t)$; note that $V>0$ for expanding sets. In view of \eqref{BV sol}, with $V$, also the weak mean curvature of $\partial^\ast E(t)$ is square integrable.
We also  refer the reader to the  work of Mugnai-Seis-Spadaro \cite{mugnai2015global} where the proof of \cite{LucStu95}  is revisited in the case of volume-preserving mean curvature flow.

\medskip
 To the best of our knowledge, the only two cases in which  assumption \eqref{conv_ass} has been shown to be satisfied \emph{a-priori} is in the graphical case \cite{logaritsch2016obstacle}, in which no singularities occur, cf.\ \cite{ecker1989mean}, and in the convex case \cite{caselles2006anisotropic}, in which no singularities appear until the solution disappears in a round point \cite{huisken1984flow}. 
 
 The main result of the present paper is to show that for a relevant class of initial data \eqref{conv_ass} holds true. The class of sets we will work with  is the class of {\em strictly mean convex sets}. Recall that a set is said to be strictly mean convex if $H>0$. Note that then, at least locally, $E$ solves a one-sided variational problem, called $\delta$-outward minimization, see Definition \ref{def H>0} below. 

More precisely, our main theorem reads as follows.
\begin{thm}\label{thm main}
Let  $E_0\subset \R^n$ be a compact set with \(C^2\) boundary and let \(n\le 7\). Assume that \(E_0\)  is strictly mean convex in the sense that $H_{\partial E_0}>0$, then \eqref{conv_ass} holds.
\end{thm}

It is easy to construct strictly mean convex sets such that the mean curvature flow starting from them develops singularities in finite time. Hence our result is the first one establishing the validity of \eqref{conv_ass} under the possible development of singularities. Note also that, to the best of our knowledge, there are no examples of initial data for which \eqref{conv_ass} does not hold.

Let us also remark that a similar question was raised by Ilmanen for the approximation of the mean curvature flow via the Allen-Cahn equation \cite[Section 13, Question 4]{ilmanen1993convergence}.

%To the best of our knowledge, the present work is the first to prove this assumption, which is not even evident in the smooth case.
%Before stating our results, we would like to point out two open problems. First, even under the additional assumption \eqref{conv_ass} the convergence to Brakke's mean curvature flow is still open.
%Second, while the scheme has a natural extension to partitions $(E_1,\ldots,E_N)$, the convergence in this more general case, e.g.\ an analogue of Theorem \ref{thm LS} is not known.
%
%\medskip
%
%If the initial conditions are mean convex, it is a well-known conjecture that \eqref{conv_ass} should hold. An early reference for such a question is as old as the work of Almgren et.\ al, namely Ilmanen's paper on the convergence of the Allen-Cahn equation \cite[Section 13, Question 4]{ilmanen1993convergence}.
%
%Our main theorem proves this conjecture for the Almgren-Taylor-Wang scheme.

Along the way we establish the following natural properties of the minimizing movements scheme \eqref{MM} for mean convex sets, which mirror Huisken's results for mean curvature flow \cite{huisken1984flow}:
\begin{itemize}
\item The sets $E_k$ are nested in the sense that $E_{k+1} \subset E_k$ for all $k\geq 1$.
\item The scheme preserves $\delta$-outward minimality and moreover, if $n\leq 7$, the minimum of the mean curvature of \(\partial E_k\),  $\min H_{\partial E_k}$ is increasing in $k$.
%\item Incidentally, our methods show that the scheme preserves (ordinary) convexity as well.
\end{itemize}
While Huisken's proofs are based on the maximum principle, our proofs are solely of variational nature.

Inspired by the work of Evans-Spruck \cite{evans1995motion} on mean curvature flow, we introduce the arrival time $u_h$ of the scheme. As the name suggests, the arrival time $u(x)$ of the mean curvature flow starting from $E_0\subset \R^n$ at a point $x\in E_0$ is the first time $t>0$ at which the flow reaches $x$, i.e., the super level set $\{u>t\}$ is equal to $ E(t) $. Similarly, as the sets $E_k$ obtained by the scheme are nested, one may also define the arrival time $u_h$ of the scheme so that $E_h(t) = \{u_h > t\}$. As one would expect,  \(u_h\) converges to \(u\), see Proposition \ref{thm:convergence}. By the  coarea formula, the proof of Theorem \ref{thm main} then boils down to the convergence of the total variation of the functions $u_h$.
This can be obtained by using a compensated compactness argument in line with the one in  \cite{evans1995motion}, together with some duality formulation of the obstacle problem established in \cite{scheven2017dual}. 
However, we also present a much simpler direct proof which is self-contained and again based on the variational principle for $u_h$.
\medskip

As an immediate consequence of our main theorem, the convergence result of Luckhaus-Sturzenhecker becomes unconditional in the case of mean convex initial data:
\begin{cor}\label{c:conv}
Suppose $n\leq 7$ and $E_0$ is strictly mean convex, then any $L^1$-limit of the approximations $E_h(t)$ is a $\BV$ solution of mean curvature flow.
\end{cor}

The paper is organized as follows. In Section \ref{sec:basic}, we establish some basic properties of strictly mean convex, so $\delta$-outward minimizing, sets and of the minimization scheme when applied to such sets.
In Section \ref{sec:arrival time} we define the arrival time of the scheme and prove that it solves an obstacle problem. In Section \ref{sec:conv} we show it converges to the arrival time of the discrete evolution and eventually in Section  \ref{sec:CC uh}, we prove Theorem \ref{thm main}.

\medskip

 \subsection*{Acknowledgements}

G.~D.~P.\ is supported by the MIUR SIR-grant ``Geometric Variational Problems" (RBSI14RVEZ). The authors would like to thank the referee for the careful reading and helpful comments which highly improved the quality of the manuscript.

\section{Basic properties of the scheme and mean convexity}\label{sec:basic}
 
We recall the definition and derive some first properties for the implicit time discretization scheme \eqref{MM} when the initial set is mean convex. The basis of our analysis is Lemma \ref{la stay mean convex}, which states that the scheme preserves mean convexity and that $\min H_{\partial E(t)}$ is non-decreasing in $t$. 
%A direct consequence is the monotonicity of the perimeters, see.\ Corollary \ref{cor monotonicity Per}.
  
  \medskip
  
Let us state  the minimization problem \eqref{MM} in a more precise language:  Given initial conditions $E_0\subset \R^n$, obtain $E_k$ for $k\in \N$ by successively minimizing $\F_h(E,E_{k-1})$:
  \begin{equation}\label{minimizing movements}
   E_k \in \arg\min \F_h(\,\cdot\,,E_{k-1}),
  \end{equation}
  where the functional $\F_h$ is given by
  \[
    \F_h(E,F) : = P(E) + \frac1h \int_{E\Delta F} d_{F}.
  \]
  Here and throughout the paper $d_F (x) := \textup{dist}(x,\partial F)$ denotes the distance function to the  boundary of $F$. We will always work with  the representative of \(F\) for which \(\overline { \partial^\ast F}=\partial F\), \( \partial^\ast F\) being the reduced boundary of \(F\), see \cite[Remark 15.3]{Maggi12}.
  
 We denote by $E_h$ the piecewise constant interpolation of the sets $E_0,E_1,E_2,\ldots $, i.e., 
  \[E_h(t)=E_k\quad \text{for }t\in[kh,(k+1)h).\]
  
  \begin{rem}\label{rem F tilde}
  It  is easy to see that the metric term $\int_{E\Delta F} d_F$ can be rewritten as
	\[
	\int_{E\Delta F} d_{F} = \int_{E} sd_{F} -\int_{F} sd_{F},
	\]
	where $sd_{F} := d_{F} - d_{\R^n\setminus F} $ denotes the signed distance function to the  boundary $\partial F$. Therefore the minimization of $\F_h(\,\cdot\,,F)$ is equivalent to minimizing
	\[
	 	P(E) + \frac1h \int_{E} sd_{F}.
	\]
	Testing \eqref{minimizing movements} with $E_{k-1}$ and summing over $k$ implies the following a priori estimate for the implicit time discretization
  \begin{equation}\label{a priori}
  	\sup_{N\geq1}  \Big\{ P(E_N) + \sum_{k=1}^N \frac1h \int_{E_k \Delta E_{k-1}} d_{E_{k-1}} \Big\} \leq P(E_0),
  \end{equation}
  which underlies Luckhaus-Sturzenhecker's compactness and conditional convergence Theorem \ref{thm LS}.
	\end{rem}
	
\begin{rem}\label{example ball}
	
In the radially symmetric case $E_0=B_{r_0}$, a Steiner symmetrization argument shows that the 
minimizers are radially symmetric. Therefore, the minimization problem \eqref{minimizing movements} reduces to finding radii $r_0 > r_1 >r_2>\ldots$ so that each $r_k$ minimizes the function
  $$
    r^{n-1} + \frac1h \int_r^{r_{k-1}}  \rho^{n-1} (r_{k-1}-\rho) d\rho.
  $$
  The Euler-Lagrange equation is
  $$
    r_k^2 -r_{k-1}\,r_k +(n-1)h=0 \textrm{  (or equivalently $\frac{r_k-r_{k-1}}{h} = - \frac{n-1}{r_k}$)},
  $$
   so that for sufficiently small $h$ the optimal radius is explicitly given by
  \[
    r_k= \frac12 \Big( r_{k-1} + \sqrt{ r_{k-1}^2 -4 (n-1) h}\Big).
  \]
  Note that for fixed $h$, after $O(r_0^{2}h^{-1})$ steps we have $r_k = 0$. Note also that, as one can easily see by induction
  \[
  r_k \ge \sqrt{r_0^2-2k(n-1)h}.
  \]
%  Moreover, the radii are monotone in the time-step size $h$ in the following sense.
%  Suppose $R=1$.  A direct but  computation shows
%  \[
%    r(r(1,h),h) \geq r(1,2h), \quad \text{i.e.,} \quad E_h(t) \subset E_{2h}(t).
%  \]
%  Note that the  inequality between the approximation and the radius of the limiting evolution  $r(t) = \sqrt{r_0^2 - 2(d-1)t}$, which reads
%  \begin{equation}\label{circle}
%   r(r,h) \geq r(h),
%  \end{equation}
%  follows nicely from a convexity argument: Substituting $t = 2(d-1)h$ and setting $f(t):= \sqrt{1-t}$, the relation \eqref{circle} is equivalent
%  to
%  \[
%   f(t) \geq \frac12\left(f(0) +f(2t)\right),
%  \]
%  which holds because $f$ is concave on $[0,1]$.
  \end{rem}

It is a well known fact in the study of mean curvature flow that mean-convexity of the initial condition (i.e. \(H_{\partial E_0}\ge 0 \)) is preserved, \cite{huisken1984flow} and that in this setting much stronger results can be obtained, see for instance  \cite{HK17,White00,White03} for an incomplete list and \cite{MetzgerSchulze08} where a  problem similar to ours is studied.

Here, as in \cite{HI01}, we introduce the variational analog of mean convexity:

  \begin{defn}\label{def H>0}
  Let $\Omega \subset \R^n$. A set $E\subset \Omega$ is called \emph{outward minimizing in $\Omega$} if
  \begin{equation}\label{H>0_Omega}
  P(E) \leq P(F)\quad \text{for all }F \text{ with } E\subset F \subset \Omega.
  \end{equation}
  If $E$ is outward minimizing in $\Omega = E+B_\delta = \{x\in \R^n \colon \dist(x,E) < \delta \} $ the $\delta$-neighborhood of $E$, then $E$ is called \emph{$\delta$-outward minimizing}, and \eqref{H>0_Omega} simply reads
  \begin{equation}\label{H>0}
  P(E) \leq P(F)\quad \text{for all }F\supset E \text{ with } \sup_{x\in F} \dist(x,E) < \delta.
  \end{equation}
  \end{defn}

\begin{rem}\label{rmk:H}
  Outward minimality as defined  above is the variational formulation of the pointwise inequality $H\geq0$. 
  It is easy to see that in our case of a smooth and strictly mean convex set $E_0$ there exists $\delta>0$ such that $E_0$ is \emph{$\delta$-outward minimizing}, see for instance \cite[Lemma 5.12]{DLMV18}.
  Note carefully that $P(E)$ denotes the perimeter in $\R^n$, not the one relative to $\Omega$.

 Each iteration of the scheme does not move further than \(O(\sqrt{h})\) in Hausdorff distance, see  \cite[Lemma 2.1,(1)]{LucStu95}, i.e., there exists a universal constant $C=C(n)$ such that
 \begin{equation}\label{eq:LS sqrt h}
 	\sup_{x\in \partial E_k} d_{E_{k-1}}(x) \leq C \sqrt{h}.
 \end{equation}
 
  \end{rem}
  \medskip
  
 Let us now recall a few basic properties of $\delta$-outward minimizing sets which will be useful in the sequel.  They are well known to experts, but for the sake of completeness we report here their simple proof, see also \cite[Section 3]{spadaro2011mean} and \cite[Section 1]{De-PhilippisPaolini09a}.
  
\begin{lem}\label{la H>0 variant}
  $E$ is outward minimizing in $\Omega$ if and only if
    \begin{equation}\label{H>0 variant}
      P(E\cap G) \leq P(G)\quad \text{for all }G\subset \Omega.
    \end{equation}
  \end{lem}
  
  \begin{proof} %[Proof of Lemma \ref{la H>0 variant}]
   We employ the basic inequality 
  \begin{equation} \label{perimeter inequality}
   P(E\cap F) + P(E\cup F) \leq P(E) + P(F).
  \end{equation}
   Given any set $G\subset \Omega$, the outward minimizing property  \eqref{H>0_Omega} of $E$ tested with $F= E\cup G$ yields 
    \[
    P(E) \overset{\eqref{H>0_Omega}}{\leq} P(E\cup G) \leq P(E) + P(G) - P(E \cap G),
    \]
    which simplifies to \eqref{H>0 variant}.
    
Vice versa, if $F\supset E$, we can apply \eqref{H>0 variant} with \(G=F\) to obtain \eqref{H>0_Omega}.

  \end{proof}
  
  A  direct consequence of this characterization is that outward minimality is stable under  $L^1$-convergence.
  \begin{cor}\label{cor H>0 stable}
    Let $E_h\to E$ in $L^1$ for some sequence $\{E_h\}_h$ of outward minimizing sets in $ \Omega$. Then $E$ is outward minimizing in $\Omega$.
  \end{cor}

  \begin{proof} %[Proof of Corollary \ref{cor H>0 stable}]
    By Lemma \ref{la H>0 variant} it is enough to show  \eqref{H>0 variant} instead of \eqref{H>0_Omega} for $E$, which in turn follows immediately from \eqref{H>0 variant} for $E_h$ and the lower semi-continuity of the perimeter.
  \end{proof}
  
  If $\Sigma(t)$ is a smooth mean curvature flow then the scalar mean curvature $H$ of $\Sigma(t)$ solves
  \[
  \partial_t H - \Delta H = \left|A\right|^2 H,
  \]
  where $A$ denotes the second fundamental form of $\Sigma(t)$ and $\Delta$ the Laplace-Beltrami operator on $\Sigma(t)$,  cf.\ \cite[Corollary 3.5]{huisken1984flow}.  In particular, if $H\geq 0$ at $t=0$, by the maximum principle $H\geq0$ for $t\geq0$ and $\min H(t)$ is non-decreasing in $t$.
  By the strong maximum principle we even have $H>0$ for $t>0$.
  
  It is well known and easy to see that $\delta$-outward minimality is preserved by the implicit time discretization \eqref{minimizing movements}, see for instance \cite{spadaro2011mean}. We report  the simple proof of this fact in the next lemma where we also establish the monotonicity of $\min H_{\partial E_h(t)}$.
  
  \begin{lem}\label{la stay mean convex}
  Let $E_0\subset\subset \Omega$ be outward minimizing in $\Omega$. Then there exists $h_0>0$ such that for all $0<h<h_0$ the implicit time discretizations $E_h$ are non-increasing in $t$, i.e., 
  \begin{equation}\label{nested}
    E_h(t) \subset E_h(s) \quad \text{for all } 0\leq s \leq t,
  \end{equation}
  $E_h(t)$ is outward minimzing in $\Omega$ for all $t\geq0$, and $E_h(t)$ solves the Euler-Lagrange equation
    \begin{equation}\label{Euler Lagrange equation}
    H_{\partial E_h(t)}(x) = \frac{d_{E_{h}(t-h)}(x)}{h}\geq0,\qquad x\in \partial^\ast E_h(t).
    \end{equation}
    Furthermore, if \(n\le 7\), $ \min H_{\partial E_h(t)} $ is non-decreasing in $t$.
  \end{lem}
  
  Note that by classical regularity for minimizers of \eqref{MM}, see e.g.\ \cite{Maggi12},  \(\partial^\ast E_h(t)\) is a \(C^2\)-manifold relatively open in \(\partial E_h(t)\) and \(\partial E_h(t)\setminus \partial^\ast E_h(t)\) has Hausdorff dimension at most \(n-8\). In particular \eqref{Euler Lagrange equation} makes sense. 
  
We also believe that the restriction \(n\le 7\) needed to show the monotonicity of $ \min H_{\partial E_h(t)} $ can be actually avoided. It seems however that this would require some version of the maximum principle for singular hypersurfaces in the spirit of \cite{Simon87}. Since however in Theorem \ref{thm LS} this restriction does not seem to be easily avoidable, we decided to restrict ourselves to this case.
  
  By Remark \ref{rmk:H}, if $E_0$ is a bounded open set of class $C^2$ with $H_{\partial E_0}>0$, there exists $\delta>0$ such that $E_0$ is outward minimizing in its $\delta$-neighborhood $\Omega=E_0+B_\delta$. 
  The smallness condition on $h$  can be dropped if $E_0$ is outward minimizing in $\R^n$.
  
    \begin{proof} %[Proof of Lemma \ref{la stay mean convex}]
    	Let $h>0$ be such that $h< h_0:= \frac1{C^2} \min_{x\in\partial \Omega}\dist^2(x,E_0)$ with $C$ from \eqref{eq:LS sqrt h}.
    	
    	Let $k\geq 1$ and assume that $E_{k-1}$ is outward minimizing in $\Omega$. 
    We first prove $E_k \subset E_{k-1}$ and then the outward minimality of $E_k$ in $\Omega$. 
    
     Since by assumption $E_{k-1}$ is outward minimizing in $\Omega$, by \eqref{eq:LS sqrt h} and our choice of $h_0$, we may employ the characterization \eqref{H>0 variant}:
    \[
      P(E_{k-1}\cap E_k) \leq P(E_k).
    \]
    We want to use $E_{k-1} \cap E_k$ as a competitor for the minimization of $\F_{h}(\,\cdot\,,E_{k-1})$.  Since
    \[
      (E_{k-1}\cap E_k ) \Delta E_{k-1} = E_{k-1} \setminus E_k \subset E_k \Delta E_{k-1}
    \]
    we have
    \[
      \frac1h\int_{(E_{k-1}\cap E_k ) \Delta E_{k-1}} d_{E_{k-1}} \leq \frac1h \int_{E_k \Delta E_{k-1}} d_{E_{k-1}}
    \]
    with strict inequality if $\L^n(E_k \setminus E_{k-1})>0$. Hence
    \[
      \F_{h}(E_{k-1}\cap E_k,E_{k-1}) \leq \F_{h}(E_k,E_{k-1})
    \] 
    with strict inequality if $\L^n(E_k \setminus E_{k-1}) >0$, which proves $E_k \subset E_{k-1}$ (up to Lebesgue null sets).
    
    Let $F$ be such that $ E_k \subset F \subset \Omega$; we want to verify $P(E_k) \leq P(F)$.  Using the outward minimality of the predecessor  $E_{k-1}$ we have
    \[
    P(F\cap E_{k-1}) \overset{\eqref{H>0 variant}}{\leq} P(F)
    \]
    and hence it is enough to prove the inequality \eqref{H>0_Omega} for sets $F$ with $E_k \subset F\subset E_{k-1}$. Using these inclusions we have
    \[
    F\Delta E_{k-1}  = E_{k-1} \setminus F \subset E_{k-1} \setminus E_k = E_k \Delta E_{k-1}
    \]
    and therefore
    \[
    \frac1h \int_{F\Delta E_{k-1}} d_{E_{k-1}} \leq  \frac1h\int_{E_k \Delta E_{k-1}} d_{E_{k-1}}.
    \]
   Now the minimality $\F_{h} (E_k,E_{k-1}) \leq \F_{h}(F,E_{k-1})$ implies $P(E_k) \leq P(F)$ and hence $E_k$ is indeed outward minimizing in $\Omega$.
    
    Since \eqref{Euler Lagrange equation} is classical, we now  turn to the proof of the monotonicity of $\inf H_{\partial E_h(t)}$. Fix $k \in \N$ and let
    \[
     x_0\in \arg\min H_{\partial E_k}.
    \]
    Since $H_{\partial E_h(t)} = \frac1h d_{E_{k-1}}$ is Lipschitz continuous and $\partial E_k$ is compact, at least one such $x_0$ exists.
    We shift $E_{k-1}$ by $h\, H_{\partial E_k}(x_0) = d_{E_{k-1}}(x_0)$ in the fixed direction $\nu_{\partial E_k}(x_0)$, i.e.,
    \[
      F_{k-1} := E_{k-1} + h\, H_{\partial E_k}(x_0)\,\nu_{\partial E_k}(x_0).
    \]
    By definition of $x_0$ we have $E_k \subset F_{k-1}$ and $x_0 \in \partial E_k \cap \partial F_{k-1}$ and, since \(n \le 7\), both boundaries are smooth in a neighborhood of \(x_0\).  Thus
    \[
     H_{\partial E_k}(x_0) \geq H_{\partial F_{k-1}}(x_0) \geq \min H_{\partial F_{k-1}} = \min H_{\partial E_{k-1}},
    \]
    which is precisely our claim.
  \end{proof}
  
  By Corollary \ref{cor H>0 stable},  limits of outward minimizing sets are  outward minimizing. From this we can easily infer the monotonicity of the perimeters. 

  \begin{cor}\label{cor monotonicity Per}
    Let $E_0\subset\subset \Omega$ be outward minimizing in $\Omega$ and let $E(t)$ be an $L^1$-limit of the implicit time discretizations $E_h(t)$. Then $E(t)$ is outward minimizing in $\Omega$  for a.e.\ $t$ and $P(E(t))$ is non-increasing in $t$.
  \end{cor}

  \begin{proof} %[Proof of Corollary \ref{cor monotonicity Per}]
    The outward minimizing  property of $E(t)$ is an immediate consequence of Lemma \ref{la stay mean convex} and Corollary \ref{cor H>0 stable}.
    Since by Lemma \ref{la stay mean convex} we have $E(t) \subset E(s)$ for $t\geq s$ we can use the mean convexity \eqref{H>0} of $E(t)$ to conclude
    $ P(E(t)) \leq P(E(s))$ for $t\geq s$.
  \end{proof}

  \medskip
 
  The basic inequality \eqref{perimeter inequality} 
  and the observation that we have the analogous equality for the distance term in $\F$ yield the general inequality
  \begin{equation}\label{general F inequality}
   \F_h(E\cap F,E_{k-1})+ \F_h(E\cup F,E_{k-1}) \leq \F_h(E,E_{k-1}) + \F_h( F,E_{k-1}).
  \end{equation}
  Therefore, if $E$ and $F$ are minimizers, so are $E\cap F$ and $E\cup F$.
  In our setting, where $E_{k-1}$ is outward minimizing , this implies the outward minimality  of all these sets and we have equality in \eqref{perimeter inequality}.

The following general lemma is a comparison result which holds independently of the initial conditions $E_0$ being mean convex and revisits Chambolle's ideas \cite{chambolle2004algorithm}.

\begin{lem}[Comparison principle, \cite{chambolle2004algorithm}]\label{lem_comparison}
	Let \(n\le 7\) and let  $E_0 , F_0 \subset \R^n$ be two bounded open sets of finite perimeter such that $E_0$ is properly contained in $F_0$ in the sense that $E_0\subset \subset F_0$.
	Let $E$ and $F$ be minimizers of $\F_h(\,\cdot\, , E_0)$ and $\F_h(\,\cdot\, , F_0)$, respectively, then $E$ is properly contained in $F$, i.e., $E\subset \subset F$.
\end{lem}

\begin{proof}
	The proof consists of two steps. First we prove the inclusion $E\subset F$, second we prove 
	$\min_{x\in \partial E} d(x,\partial F) >0$.  
	
	Inasmuch as $E_0\subset \subset F_0$, the boundaries have a definite distance
	$\min_{x\in \partial E_0} d(x,\partial F_0) >0$, which implies the strict inequality
	\begin{equation}\label{f<g}
		 sd_{F_0} <sd_{E_0} \quad \text{in } \R^n.
	\end{equation}	
	Probing the minimality of $E$ and $F$ for the modified functionals in Remark \ref{rem F tilde}  with $E\cap F$ and $E\cup F$, respectively, yields
	\[
		P(E) +\frac1h \int_E sd_{E_0} \leq P(E\cap F) + \frac1h \int_{E\cap F} sd_{E_0}
	\]
	and
	\[
		P(F) +\frac1h \int_F sd_{F_0} \leq P(E\cup F) + \frac1h \int_{E\cup F} sd_{F_0}.
	\]
	Summing these two inequalities and using the general inequality for the perimeter of intersections and unions of sets \eqref{perimeter inequality} we obtain
	\[
		\int_E sd_{E_0} + \int_F sd_{F_0}  \leq   \int_{E\cap F} sd_{E_0} +\int_{E\cup F} sd_{F_0}.
	\]
	Rearranging the terms and using the obvious identities $\chi_{E\cap F} = \chi_E \chi_F$ and $\chi_{E\cup F} = \chi_E + \chi_F - \chi_E\chi_F$ along the way, we obtain
	\[
	0\leq \int \left( sd_{E_0} - sd_{F_0}\right) \chi_E \left( 1 - \chi_F\right) 
	= \int_{E\setminus F}  \left( sd_{E_0} - sd_{F_0}\right).
	\]
	Since by \eqref{f<g} the integrand is strictly negative, this means that $\L^n(E\setminus F)=0$
	and hence $E\subset F$.
	
	Now assume for a contradiction $\partial E \cap \partial F \neq \emptyset$. Let $x_0 \in \partial E \cap \partial F$ be a point in the intersection. Since $E \subset F$ we have $H_{\partial E} \geq H_{\partial F}$ at that point $x_0$ and therefore
	\[
		\frac1h sd_{E_0} = -H_{\partial E} \leq - H_{\partial F} = \frac1h sd_{F_0},
	\]
	a contradiction to \eqref{f<g} (note that as in the proof of Lemma \ref{la stay mean convex} we have used the restriction \(n\le 7\) to ensure smoothness of the boundaries at the touching point).
	\end{proof}

  \section{The arrival time for the implicit time discretization}\label{sec:arrival time}

Since by Lemma \ref{la stay mean convex} the sets $E_h(t)$ are nested, we can define the (discrete) arrival time \(u_h\) for the scheme. In this section we show that, up to subsequences, \(u_h\) converges uniformly to some continuous function  \(u\). In the next section we will identify \(u\) as the arrival  time for the limiting evolution starting from \(E_0\).

\begin{defn}\label{dfn:at}
	Let $E_0$ be outward minimizing in the sense of Definition \ref{def H>0}, let $E_k$, $k\geq1$, be given by \eqref{minimizing movements} and let $E_h$ denote their piecewise constant interpolation in time. We define the \emph{arrival time} $u_h\colon \R^n \to [0,\infty)$ by
	\begin{equation}\label{def uh}
		u_h(x) := h\sum_{k\geq 0} \chi_{E_k}(x) = \int_0^\infty \chi_{E_h(t)}(x)\,dt \quad (x\in \R^n).
	\end{equation}
\end{defn}

Let us first note that $u_h \in \BV(\R^n)$ since the a priori estimate \eqref{a priori} implies
	\begin{equation}\label{a priori uh}
	\int_{\R^n}\left| Du_h\right| = \int_0^{T_h} P(E_h(t))\,dt \leq T_h P(E_0),
	\end{equation}
	where $T_h$ denotes the extinction time of $(E_h(t))_{t\geq0}$. Note that  the extinction time is finite: If $R>0$ is sufficiently large such that $E_0\subset B_R$, then by Lemma \ref{lem_comparison} we have $E_h(t) \subset B_{r_h(t)}$, where $r_h$ is given in Remark \ref{example ball} and satisfies $r_h(t)=0$ for $t$ larger than $O(R^2)$.
	
	\medskip
	
The following lemma  states that for our mean convex initial condition, the arrival time solves a (one-sided) variational problem.
\begin{lem}\label{lem uh outward minimizing}
 	 Let $E_0\subset\subset \Omega$ be outward minimizing in $\Omega$ in the sense of Definition \ref{def H>0}. Then there exists $h_0>0$ such that for $0<h<h_0$, the arrival time $u_h$ is outward minimizing in $\Omega$ in the sense that
 	 \begin{equation}\label{eq:uh outward minimizing}
 	 	\int_{\R^n} \left| D u_h \right| \leq \int_{\R^n} \left|D v\right| \quad \text{for all } v\in \BV(\R^n) \text{ s.t. } v\geq u_h \text{ and } v=0 \text{ in }\R^n\setminus \Omega.
 	 \end{equation}
  \end{lem}

	Again, the smallness condition on $h$ can be dropped in case of $\Omega=\R^n$.
	
\begin{proof}
	Given $v \in \BV(\R^n)$ with $v\geq u_h$ and $v=0$ in $\R^n\setminus \Omega$ we employ the coarea formula, cf.\ \cite[Theorem 3.40]{ambrosio2000functions}, to manipulate the total variation of $v$:
	\[
		 \int_{\R^n} \left|D v\right| = \int_0^\infty P(\{ x\in \R^n \colon v(x)> t \}) \,dt.
	\] 
	Since $v\geq u_h$ and $v=0$ in $\R^n\setminus \Omega$ imply
	\[
		E_h(t) = \{ x\in \R^n \colon u_h(x)> t \} \subset \{ x\in \R^n \colon v(x)> t \} \subset \Omega,
	\]
	the super level sets of $v$ are admissible for \eqref{H>0_Omega} and we obtain
	\[
		\int_{\R^n} \left|D u_h\right| = \int_0^\infty P(E_h(t) ) \,dt 
		\leq \int_0^\infty P(\{ x\in \R^n \colon v(x)> t \}) \,dt =  \int_{\R^n} \left|D v\right| . \qedhere
	\]
\end{proof}

The next lemma states that we have a uniform estimate on the modulus of continuity of $u_h$ except for fluctuations on scales below $ h$; and hence after passing to a subsequence, we obtain uniform convergence to a continuous function.

\begin{lem}\label{lm:conv}
	Let \(n\le 7\) and let  $E_0$ be a  bounded open set of class $C^2$ with $ H_{\partial E_0} >0$. Then there exists a subsequence $h_j\downarrow0$ and a continuous function $u\colon \R^n\to [0,\infty) $ with $\supp u \subset \bar{E_0}$ such that
	\begin{align}\label{uniform convergence}
 u_{h_j} &\to u & & \text{uniformly} \\
	 D u_{h_j} &\rightharpoonup  Du & &  \text{as measures}  \label{gradient weak convergence}
	\end{align}
\end{lem} 

\begin{proof}
	Let $H_0 := \min H_{\partial E_0} >0$, which by Lemma \ref{la stay mean convex} implies 
	$\min H_{\partial E_k} \geq H_0$ for all $k\geq0$.
	
	We claim that we have a uniform bound on the modulus of continuity up to fluctuations on scales below $ h$, i.e.,
	\begin{equation}\label{mod of cont}
		\left| u_h(x)-u_h(y) \right| \leq \frac{1}{H_0} |x-y| + h \quad \text{ for all } x,y\in \R^n.
	\end{equation}		
	In order to prove \eqref{mod of cont} let $x,y\in E_0$ be given. 
	Without loss of generality  we may assume $x \in E_n$ and $y\in E_{m}$ with $-1 \leq m < n$, where we have set  $E_{-1} := \R^n \setminus E_0$. 
	Since  the sets $E_k$, $k\geq0$ are nested, the segment $[x,y] $ intersects the intermediate boundaries non-trivially: There are points $z_k$, $k=m+1,\dots,n$, such that $z_k \in \partial E_k \cap [x,y]$. Using the Euler-Lagrange equation \eqref{Euler Lagrange equation} along these points we obtain
	\[
		\left| x-y \right| \geq\left| z_{n} - z_{m+1} \right|  
		=  \sum_{k=m+2}^{n} \left| z_k - z_{k-1} \right| 
		\geq \sum_{k=m+2}^{n} d( z_k, \partial E_{k-1})
		\geq (m-n-1) h  H_0.
	\]
	Since $|u(x)-u(y)| = (m-n)h$, this is precisely our claim \eqref{mod of cont}.
	Therefore, by Arzel\`{a}-Ascoli, we obtain the compactness \eqref{uniform convergence}. The weak convergence of the gradients \eqref{gradient weak convergence} follows immediately from the uniform bound \eqref{a priori uh}.
\end{proof}

\section{Convergence to the continuous arrival time}\label{sec:conv}

Let \(E_0\) be an outward minimizing set such that \(H_{\partial E_0}>0\). According to the  previous section the  arrival times \(u_h\) of the discrete  scheme converge, up to subsequences, to a limiting function \(u\). In this section we identify this function as the arrival time of the limiting equation. We start by recalling the following

\begin{thm}[Evans-Spruck \cite{evans1995motion}]\label{thm:es}	
Let \(E_0\) be a bounded open set of class $C^2$ with \(H_{\partial E_0}>0\). Then there exists a unique continuous viscosity solution \(u\) of 
\begin{equation}\label{e:at}
\begin{cases}
 |Du|\diver\Bigg(\dfrac{Du}{|Du|}\Bigg)=-1\qquad &\textrm{in \(E_0\)}\\
u=0&\textrm{on \(\partial E_0\).}
\end{cases}
\end{equation}
Moreover,   for all \(t\in[0, \sup u]\) the set \(\{u\ge t\}\) is the evolution of  \(\overline{E_0}=\{u\ge 0\}\) via mean curvature flow.
\end{thm}

Here  a solution of \eqref{e:at} is understood in the viscosity sense, that is for all \(x\in E_0\) and all  \(\varphi \in C^2(E_0)\) such that \(u-\varphi\) has a minimum at \(x\) (resp. a maximum) then 
\begin{align}
\Delta \varphi(x) -\frac{D^2 \varphi(x)[D\varphi(x),D\varphi(x)]}{|D\varphi(x)|^2}\le -1 \quad(\ge -1) \qquad&\textrm{if \(D\varphi(x)\ne 0\)} \label{e:sup1}
\\
\exists \eta\in \mathbb S^{n-1} \textrm{ such that } \Delta \varphi(x)-D^2\varphi(x)[\eta, \eta] \le -1\quad (\ge -1)\qquad&\textrm{if \(D\varphi(x)= 0\).}\label{e:sup2}
\end{align}

The following proposition is the elliptic analog of \cite{chambolle2004algorithm}, see also \cite[Section 7]{chambolle2015nonlocal}, and shows that the discrete arrival time converges to the (unique) viscosity solution of  \eqref{e:at}.

\begin{prop}\label{thm:convergence}

Let \(E_0\) be as in Theorem \ref{thm:es} and let \(u_h\) be as in Definition \ref{dfn:at}. Then every limit point of \(u\) of \(u_h\) is a viscosity solution of \eqref{e:at}. In particular  the whole sequence \(u_h\) converges to \(u\).
\end{prop}

\begin{proof}
Let \(u\) be such that (up to subsequences) \(u_{h}\to u\) uniformly.  Let \(x\in E_0\)  and  \(\varphi \in C^{2}(E_0)\) be such that \(u-\varphi\) has a minimum at \(x\). By changing coordinates we may assume without loss of generality that \(x=0\), moreover, by replacing \(\varphi\) by \(\varphi-C|x|^4\) we may assume that the minimum is global and strict:
\begin{equation}\label{minstrict}
u(x)-\varphi(x)>u(0)-\varphi(0) \qquad \textrm{for all \(x\in E_0\setminus\{0\}\).}
\end{equation}
By classical arguments we can find a sequence of points \(x_h\) such that \(x_h \to 0\) and 
\[
(u_h)_*(x)-\varphi(x)\ge (u_h)_*(x_h)-\varphi(x_h)
\]
where \((u_h)_*\) is the lower semicontinuous envelop of \(u_h\), namely
\[
(u_h)_*=\sum_{k=1}^{T_h/h} h \chi_{\Int(E_k)}.
\]
Here \(T_h\) is the extinction time of the scheme. Note in particular that \( (u_h)_*\to u\) uniformly. For simplicity, from now on we assume that the sets \(E_k\) are open and that \(u_h\) is already lower-semicontinuous  (observe  that by the regularity theory for almost minimizers of the perimeter \(|\overline{E_k} \setminus \Int(E_k)|=0\) which allows us to choose such a representative). We also let \(k_h\in \mathbb N\) be the unique integer such that \(u_h(x_h)=k_h h\). In particular \(x_h \in E_{k_h}\).

 We now distinguish two cases.

\medskip
\noindent 
{\em Case 1: \(D\varphi(0)\ne 0\).} Since \(x_h\to 0\) we have \(D\varphi(x_h)\ne 0\) if \(h\) is sufficiently small.  Hence, \(u_h\) cannot be flat constant in a neighborhood of \(x_h\), so \(x_h \notin \Int(E_{k_h}\setminus E_{k_h+1})\) and thus, since \(E_{k_h}\) is open,
\[
x_h \in \partial E_{k_h+1}.
\] 
In particular 
\[
U:=\big\{ \varphi>\varphi (x_h)\} \subset E_{k_h+1} \qquad\textrm{and}\qquad x_h \in \partial U\cap\partial E_{k_h+1}.
\]
Since both \(\partial U\) and \(\partial E_{k_h+1}\) are smooth in a neighborhood of \(x_h\), the comparison principle and the Euler-Lagrange equation \eqref{Euler Lagrange equation} yield
 \begin{equation}\label{e:dis}
 \begin{split}
 \diver\Bigg(\frac{D\varphi(x_h)}{|D\varphi(x_h)|}\Bigg)&=-H_{\partial U}(x_h)\le -H_{\partial E_{k_h+1}}(x_h)
 \\
 &= -\frac{\dist (x_h, \partial E_{k_h})}{h}\le -\frac{\dist (x_h, \partial \{\varphi> \varphi(x_h)-h\})}{h},
 \end{split}
 \end{equation}
 where in the last inequality we have used that
 \[
 E_{k_h}^c=\{u\le u(y_h)-h\}\subset \{\varphi\le \varphi(x_h)-h\}.
 \]
Moreover, by Taylor expansion, one easily verifies 
\begin{equation}\label{e:taylor}
\frac{\dist (x_h, \partial \{\varphi> \varphi(x_h)-h\})|D\varphi(x_h)|}{h}\to 1\qquad\textrm{as \(h\to 0\)}.
\end{equation}
Combining \eqref{e:dis} and \eqref{e:taylor} we conclude the validity of \eqref{e:sup1}.

\medskip
\noindent
{\em Case 2: \(D\varphi(0)=0\).} This time we can not assume a priori that \(D\varphi (x_h)\neq 0\). To overcome this difficulty we exploit Jensen's $\inf$-convolution (on a fixed scale of order \(h\)). To this aim let us define
\[
v_h(x):=\inf_{y\in \overline{E_0}} \Big\{u_h(y)+\frac{ |x-y|^4}{2c^4_nh}\Big\} \quad \text{for \(x\in E_0\)},
\] 
where \(c_n\) is a constant that will be fixed later in dependence only on the dimension \(n\). We also let \(z_h\) be a minimum point of \(v_h-\varphi\), namely
\[
v_h(x)-\varphi(x)\ge v_h(z_h)-\varphi(z_h) \quad \text{for all } x\in \bar{E_0}
\]
and let \(y_h\in \overline{E_0}\) be such that 
\[
v_h(z_h) =u_h(y_h)+\frac{|z_h-y_h|^4}{2c^4_nh}.
\]
Note that the existence of \(y_h\) is ensured by the lower semicontinuity of \(u_h\).

We now divide the proof in some  steps:

\medskip
\noindent
{\em Step 1: \(|z_h-y_h| \to 0\).}  Indeed, since \(v_h\le u_h\le 2 \|u\|_{\infty}\) we obtain
\[
|z_h-y_h|^4\le 8c^4_n h \|u\|_{\infty}\to 0 \quad\textrm{as \(h\to 0\)}.
\]

\medskip
\noindent
{\em Step 2: \(z_h \to 0\).} Indeed, by \(v_h\le u_h\) and the definition of $z_h$,
\[
u_h(x_h)-\varphi(x_h) \ge v_h(z_h)-\varphi(z_h)\ge u_h(y_h)-\varphi(y_h)+\varphi(y_h)-\varphi(z_h).
\]
If we let \(\bar z\in \overline {E_0} \) be an accumulation point  of \(z_h\) (and hence of \(y_h\)) we deduce from the above inequality and the uniform convergence of \(u_h\) to \(u\)  that
\[
u(0)-\varphi(0)\ge u(\bar z)-\varphi (\bar z)
\] 
which in view of \eqref{minstrict} forces \(\bar z=0\).

\medskip
\noindent
{\em Step 3: \(z_h \ne y_h\).} Let us assume by contradiction that \(z_h=y_h\). By the very definition of \(v_h\) this means that 
\begin{equation}\label{e:bound}
u_h(z_h)=v_h(z_h) \le u(y)+\frac{|y-z_h|^4}{c^4_n h} \quad \text{for all } y\in \bar{E_0}.
\end{equation}
Let also  \(j_h\in \N\) be such that \(u_h(z_h)=j_h h\). Note that since \(u>0\) in \(E_0\) and \(u_h(z_h)\to u(0)>0\) we may assume that \(j_h \gg 1\). In particular
\[
z_h\in E_{j_{h}}\setminus E_{j_{h}+1}.
\]
We now note that \eqref{e:bound} implies
\begin{equation}\label{e:inc}
F_0:=B\big(z_h,c_n \sqrt{h}\big)\subset\subset \big\{u_h\ge (j_h -1)h\big\}=E_{j_h-1}.
\end{equation}
If we let \(F_1\) and \(F_2\) be minimizers of \eqref{minimizing movements} starting from \(F_0\) and \(F_1\), respectively, Remark \ref{example ball} ensures that 
\[
F_2=B\big(z_h,r_h\big) \qquad\textrm{with} \qquad  r_h\ge \sqrt{c_n-4(n-1)}>0
\]
provided \(c_n\) is chosen sufficiently large. However, by  Lemma \ref{lem_comparison} and \eqref{e:inc}
\[
z_h \in F_2\subset\subset E_{j_h+1},
\]
a contradiction.

\medskip
\noindent
{\em Step 4: Conclusion.} By the very definitions of \(v_h\), \(y_h\) and \(z_h\) we have
\begin{equation}\label{e:key}
u_h(y_h)+\frac{|z_h-y_h|^4}{2c^4_n h}-\varphi(z_h)\le  u_h(y)+\frac{|x-y|^4}{2c^4_n h}-\varphi(x) \quad \text{for all } x,y\in \bar{E_0}.
\end{equation}
In particular, the optimality condition in the \(x\)-variable implies
\[
D \varphi(z_h)=\frac{2|z_h-y_h|^2(z_h-y_h)}{c_n^4 h} \ne 0.
\]
Moreover, if we set
 \[ 
 \psi_h(x):=\varphi(x+(z_h-y_h))+\frac{|z_h-y_h|^4}{2c^4_n h},
\]
the function \(u-\psi_h\) has a minimum at \(y_h\)  with \(D \psi_h(z_h)\ne 0\). By the very same arguments of Case 1 we obtain that
\[
\Delta \varphi(z_h) -\frac{D^2 \varphi(z_h)[D\varphi(z_h),D\varphi(z_h)]}{|D\varphi(z_h)|^2}\le -1+o(1),
\]
which gives \eqref{e:sup2} with \(\eta\) being any limiting point of the sequence \(\frac{D\varphi(z_h)}{|D\varphi(z_h)|}\).

Since the case in which \(u-\varphi\) has a maximum at some \(x\in E_0\) can be treated analogously, this completes the proof.
\end{proof}

\section{Compensated compactness for the arrival time and proof of Theorem~\ref{thm main}}\label{sec:CC uh}

In this section we establish the convergence of the total variations of the arrival times $u_h$ and prove Theorem \ref{thm main}.  Our proof is elementary and only uses the variational principle for $u_h$ established in Lemma \ref{lem uh outward minimizing}. We also state a second proof which seems more robust and might be applicable to similar problems. This second proof based on the compensated compactness argument of  Evans-Spruck \cite{evans1995motion} together with the dual problem of the variational principle for $u_h$ viewed as an obstacle problem for \(BV\) functions established in \cite{scheven2017dual}.

\begin{prop}\label{prop CC uh}
%:
	Let $E_0$ be strictly mean convex in the sense of Definition \ref{def H>0} and let $u_{h_j}$ defined by \eqref{def uh} satisfy \eqref{uniform convergence} and \eqref{gradient weak convergence}.
	Then, 
	\[
		\left| Du_{h_j}\right| \rightharpoonup  \left| Du \right| \quad \text{as measures.}
	\]
	In particular it holds
	\[
		\int_{\R^n} \left| Du_{h_j}\right| \to \int_{\R^n}  \left| Du \right|.
	\]
\end{prop}

While the compensated compactness argument of Evans-Spruck is based on the curious estimate 
\begin{equation}\label{EvansSpruck}
\sup_{\varepsilon>0} \int_{\R^n} |H_\varepsilon(x)| \,dx < \infty,
\end{equation}
which miraculously holds true for the elliptic regularizations $u_\varepsilon$ of the level set formulation, this estimate is very intuitive in our situation:

Informally, the Euler-Lagrange equation of the minimization problem in Lemma \ref{lem uh outward minimizing} reads 
	\[
		- \diver \left(  \frac{Du_h}{|Du_h|}\right) \geq 0.
	\]
	This means that these distributions are in fact measures, for which it should be reasonable to get appropriate bounds. This resembles the $L^1$-bound \eqref{EvansSpruck} and would allow us to pass to the limit in 
	\begin{equation}\label{heuristics}
	\int \zeta \left| Du_h\right| = \int  \zeta Du_h \cdot  \frac{Du_h}{|Du_h|} = -\int \zeta u_h 
	\diver \left(  \frac{Du_h}{|Du_h|}\right) - \int u_h \,\frac{Du_h}{|Du_h|} \cdot D\zeta.
	\end{equation}
	This argument can be made rigorous, see Remark \ref{rmk_oldproof} below. Let us first show a simpler direct proof.

	\begin{proof}[Proof of Proposition \ref{prop CC uh}]
		By lower semi-continuity, we only need to prove the inequality
		\begin{equation}\label{eq:proof of prop CC uh}
			\limsup_{h\downarrow0} \int \left| Du_h\right| \leq \int \left| Du\right|.
		\end{equation}
		Since $u_h\to u$ uniformly, for any $\varepsilon>0$ there exists $h_0>0$ such that
		\[
			u_h<u+\varepsilon \quad \text{whenever } 0<h<h_0.
		\] 
		Multiplying $u+\varepsilon$ with a cutoff $\eta\in C_0^\infty(\Omega)$ of $E_0$ in $\Omega$, $v=(u+\varepsilon)\eta$ is an admissible competitor for \eqref{eq:uh outward minimizing}, so that
		\[
			 \int \left| Du_h\right| 
			 \stackrel{\eqref{eq:uh outward minimizing}}{\leq}  
			 \int \left|Dv\right|
			 = \int \eta \left| D u\right| + \int (u+\varepsilon)\left|D\eta \right| 
			 \leq   \int \left| D u\right| + \varepsilon \int \left|D\eta \right|,
		\]
		where we have used $\eta \leq 1$ for the first, and $u+\varepsilon = \varepsilon $ on $\supp \eta \subset  \Omega \setminus E_0$ for the second right-hand side term. Passing first to the limit $h\downarrow0$ and then $\varepsilon\downarrow0$ yields \eqref{eq:proof of prop CC uh}.
		\end{proof}
	
\begin{rem}\label{rmk_oldproof}
	For the alternative proof of Proposition \ref{prop CC uh}, which makes the compensated compactness argument \eqref{heuristics} rigorous, we interpret the minimization problem in Lemma \ref{lem uh outward minimizing} as an obstacle problem in a $\delta$-neighborhood $\Omega$ of $ E_0$ with homogeneous Dirichlet boundary conditions. Here the obstacle is of class BV and happens to be our minimizer $u_h$ itself. This allows us to use the general theory for dual formulations of obstacle problems: By \cite[Theorem 3.6, Remark 3.8]{scheven2017dual} the dual problem reads
	\[
		\max  \llbracket\sigma, Du_h^+ \rrbracket(\bar{\Omega}).
	\]
	where the maximum runs over all measurable vector fields $\sigma\colon \Omega\to \R^n$ with $|\sigma|\leq 1$ a.e. \ in $\Omega$ and $\diver \sigma \leq 0$ distributionally in $\Omega$.
	Note that this implies that $\diver \sigma$ is a measure on $\bar{\Omega}$ and
	\begin{equation}\label{diver sigma bounded}
		(-\diver \sigma)(\bar \Omega) \leq \H^{n-1}(\partial \Omega).
	\end{equation}
	 Here 
	 \[
	 u_h^+(x)=\aplimsup_{y\to x} u_h(y)
	 \]
	  denotes the largest representative of $u_h$, see \cite{scheven2017dual}, and the measure $\llbracket\sigma, Du_h^+ \rrbracket$ is defined as 
	\begin{equation}
		\llbracket\sigma, Du_h^+ \rrbracket (\zeta) := -\int_{\Omega} \zeta u_h^+ \diver(\sigma)\, dx
		- \int_{\Omega} u_h \left( \sigma \cdot D \zeta\right) dx,
	\end{equation}
	for test functions $\zeta\in C^1(\bar{\Omega})$.
	This yields a vector field $\sigma_h$ for any $h>0$ with the above mentioned properties and such that
	\begin{equation}\label{duality equation uh}
		\int \left| Du_h\right| 
		=\llbracket\sigma, Du_h^+ \rrbracket(\bar{\Omega}) = \llbracket\sigma, Du_h^+ \rrbracket(\R^n).
	\end{equation}
	Here we used the fact that $u_h $ vanishes away from $E_0 \subset\subset \Omega$. 
	
	Since $|\sigma_h|\leq 1 $, we may assume that there exists a measurable vector field $\sigma$ with $|\sigma|\leq1$ such that
	\begin{equation}
		\sigma_{h_j} \overset{\ast}{\rightharpoonup} \sigma\quad \text{in }L^\infty.
	\end{equation}
	Moreover, by \eqref{diver sigma bounded}, there exists a subsequence, which we do not relabel, and  a measure $\mu$ such that
	\begin{equation}\label{diver sigma convergence}
		\diver \sigma_{h_j} \rightharpoonup \mu \quad \text{as measures}.
	\end{equation}
	In particular
 	\begin{equation}
 	\diver \sigma = \mu \quad \text{  in }\Omega.
 	\end{equation}

	Now we can make the idea of the aforementioned compensated compactness argument rigorous.
	By \eqref{duality equation uh} we have
	\begin{equation}\label{Du rigorous}
		\int \zeta \left| Du_h\right|
		= -\int \zeta u_h^+ \diver(\sigma_h)\,dx
		- \int u_h \left( \sigma_h \cdot D \zeta\right) dx,
	\end{equation}
	which is precisely the analogue of \eqref{heuristics} with the important difference that we can give a meaning to (and have precise estimates for) all products appearing on the right.
	Along the subsequence $h_j\downarrow0$, on the one hand, since $u=\lim u_{h_j}$ is continuous, we have
	\[
		\lim_{h_j\downarrow0}-\int \zeta u \diver(\sigma_{h_j})\,dx 
		\stackrel{\eqref{diver sigma convergence}}{=} - \int \zeta u \,d\mu.
	\]
	On the other hand, by the uniform convergence \eqref{uniform convergence}, we have
	\[
		\left| -\int \zeta\, \big( u_{h_j}^+ -u\big) \diver(\sigma_{h_j})\,dx \right|
		\stackrel{\eqref{diver sigma bounded}}{\leq} \|\zeta\|_\infty \| u_{h_j}^+ - u\|_\infty \H^{n-1}(\partial \Omega) \to 0.
	\]
	Therefore, we can pass to the limit in the first right-hand side product of \eqref{Du rigorous}:
	\begin{equation}
		\lim_{h_j\downarrow0} -\int \zeta u_{h_j}^+ \diver(\sigma_{h_j})\,dx
		= - \int \zeta u \,d\mu = \int D(\zeta u) \cdot \sigma \,dx.
	\end{equation}
	Since $\supp u_h \subset \Omega$ is equibounded, the convergence $u_{h_j} \to u$ is strong in $L^1$ and hence we may pass to the limit in the second right-hand side product of \eqref{Du rigorous}. Therefore, for any non-negative test function $\zeta\in C^1(\R^n)$ we obtain
	\[
			\lim_{h_j\downarrow0}\int \zeta \,| Du_{h_j}|
			= \int D(\zeta u) \cdot \sigma \,dx
			-\int u \left( \sigma \cdot D \zeta\right) dx
			= \int \zeta\, \sigma \cdot Du
			\leq \int \zeta \left| Du\right|,
	\]
	where we used the pointwise bound $\left| \sigma \right| \leq 1$ a.e.\ in the last inequality.
	The lower semicontinuity of the total variation implies
	\[
	\int \zeta \left| Du\right|\leq \liminf_{h_j\downarrow0}\int \zeta \left| Du_{h_j}\right|
	\]
	for all non-negative test function $\zeta\in C^1(\R^n)$.
	Therefore 
	\[
	\lim_{h_j\downarrow0}\int \zeta \left| Du_{h_j}\right| = \int \zeta \left| Du\right|
	\]
	holds for all non-negative test functions $\zeta \in C^1(\R^n)$. By linearity and continuity in $\zeta$ the convergence holds for all continuous test functions $\zeta\in C( \R^n)$ without restriction on the sign, which proves $\left| Du_{h_j}\right| \rightharpoonup  \left| Du \right| $ as measures.
\end{rem}

We are now ready to prove Theorem \ref{thm main}:

\begin{proof}[Proof of Theorem \ref{thm main}]
	 Passing to a subsequence, we may assume $E_{h_j} \to E $ in $L^1$. By Proposition \ref{thm:convergence}  \(u_{h_j}\) converges  to the arrival time of the limiting evolution \(u\). By the co-area formula and Proposition \ref{prop CC uh} 
	 \[
\lim_{h\to 0} \int_0^\infty P(E_{h}(t)) \,dt = \lim_{h \to0} \int_{\R^n} 	|Du_{h}| = \int_{\R^n} |Du|=\int_0^\infty P(E(t))\,dt,
	 \]
	which proves \eqref{conv_ass}.
\end{proof}
	
%\bibliographystyle{acm}
%\bibliography{lit}

  \end{document}